\documentclass[letterpaper]{amsart}

\usepackage[letterpaper,left=1.5in,right=1.5in,top=1in,bottom=1in]{geometry}

\usepackage{ytableau}

\newcommand{\myauthor}{Benjamin Antieau and Ben Williams}
\newcommand{\mytitle}{Prime decomposition for the index of a Brauer class}

\title{\mytitle}
\author{\myauthor}
\date{}

\usepackage[pdfstartview=FitH,
            pdfauthor={\myauthor},
            pdftitle={\mytitle},
            colorlinks,
            linkcolor=reference,
            citecolor=citation,
            urlcolor=e-mail,
            ]{hyperref}

\usepackage{amsmath}
\usepackage{amscd}
\usepackage{amsbsy}
\usepackage{amssymb}
\usepackage{verbatim}
\usepackage{eufrak}
\usepackage{microtype}
\usepackage{hyperref}
\usepackage{mathrsfs}
\usepackage{amsthm}
\usepackage[all,cmtip]{xy}
\usepackage{tikz}
\usetikzlibrary{matrix,arrows}
\usepackage[abbrev,lite]{amsrefs}

\usepackage{mathpazo}

\usepackage{color}
\definecolor{todo}{rgb}{1,0,0}
\definecolor{conditional}{rgb}{0,1,0}
\definecolor{e-mail}{rgb}{0,.40,.80}
\definecolor{reference}{rgb}{.20,.60,.22}
\definecolor{mrnumber}{rgb}{.80,.40,0}
\definecolor{citation}{rgb}{0,.40,.80}

\DeclareMathAlphabet{\mathpzc}{OT1}{pzc}{m}{it}
\usepackage{dsfont}

\DeclareMathOperator{\End}{End}

\DeclareMathOperator{\PGL}{PGL}

\DeclareMathOperator{\GL}{GL}

\DeclareMathOperator{\Hoh}{H}

\DeclareMathOperator{\Br}{Br}

\DeclareMathOperator{\per}{per}

\DeclareMathOperator{\ind}{ind}

\DeclareFontEncoding{OT2}{}{}

\newcommand{\iso}{\cong}

\newcommand{\CC}{\mathds{C}}

\newcommand{\Gm}{\mathds{G}_{m}}

\newcommand{\op}{\mathrm{op}}

\newcommand{\Ascr}{\mathscr{A}}
\newcommand{\Bscr}{\mathscr{B}}

\newcommand{\Oscr}{\mathscr{O}}

\let\oldmarginpar\marginpar
\renewcommand\marginpar[1]{\-\oldmarginpar[\raggedleft\footnotesize #1]%
{\raggedright\footnotesize #1}}

\newcommand{\tensor}{\otimes}

\newcommand{\Mrm}{\mathrm{M}}

\theoremstyle{plain}
\newtheorem{theorem}{Theorem}
\newtheorem*{theorem*}{Theorem}
\newtheorem{lemma}[theorem]{Lemma}

\newtheorem{proposition}[theorem]{Proposition}

\newtheorem{corollary}[theorem]{Corollary}

\newtheoremstyle{named}{}{}{\itshape}{}{\bfseries}{.}{.5em}{#1 \thmnote{#3}}
\theoremstyle{named}

\theoremstyle{definition}

\theoremstyle{remark}

\begin{document}

\begin{abstract}
    We prove that the index of a Brauer class satisfies prime decomposition
    over a general base scheme. This
    contrasts with our previous result that there is no general prime
    decomposition of Azumaya algebras.

    \vspace{10pt}

    \vspace{10pt}
    \paragraph{Mathematics Subject Classification 2010.}
    Primary:
    \href{http://www.ams.org/mathscinet/msc/msc2010.html?t=14Fxx&btn=Current}{14F22},
    \href{http://www.ams.org/mathscinet/msc/msc2010.html?t=18Exx&btn=Current}{16K50}.
\end{abstract}

\maketitle

\section{Introduction}

The Brauer group $\Br(k)$ of a field $k$ classifies the central simple
algebras over $k$ up to Brauer equivalence. Two such algebras $A$ and $B$ are Brauer
equivalent if $\Mrm_m(A)\iso\Mrm_n(B)$ for some integers $m,n$, where $\Mrm_n(A)$
denotes the $n\times n$ matrix ring with coefficients in $A$. The class of $A$
in $\Br(k)$ will be written as $[A]$. The group structure is given by tensor
product of algebras, and $-[A]=[A^{\op}]$. The Brauer group is a
key arithmetic invariant of the field $k$ and has been studied for about a
century.

A couple of facts about $\Br(k)$ are relevant to this paper. The first is that
it is a torsion group: each element $\alpha$ has a finite order, called the
\emph{period}, denoted $\per(\alpha)$. In terms of Brauer equivalence, $\per([A])$ is the least
positive integer $m$ such that $A^{\otimes m}\iso\Mrm_n(k)$. The second is
that $\dim_k A=d^2$ for each central simple algebra $A$. The number $d$
appearing in this equation is called the \emph{degree} of $A$.

Given $\alpha\in\Br(k)$, a theorem of Wedderburn implies that $\alpha=[D]$ for
a unique division algebra $D$. Classically, the \emph{index} of $\alpha$,
written $\ind(\alpha)$, is defined to be the degree of $D$. The authors
discovered in~\cite{aw4} that this definition is unsuitable for
generalization to the Brauer group to a scheme, a topological space, or more
generally a locally ringed topos. Rather, we define
\begin{equation*}
    \ind(\alpha)=\gcd\{\deg(A)|[A]=\alpha\}.
\end{equation*}
It is well-known that this definition agrees with the previous one given for
fields. For this and all other standard facts about the Brauer group and
division algebras used in this introduction,
see for example the book of Gille and Szamuely~\cite{gille-szamuely}*{Section 4.5}.

A related classical fact is that the index $\ind(\alpha)$ may be computed as the
minimum degree of a separable splitting field $K/k$ that splits $\alpha$. This
is an important connection because it allows a direct Galois-theoretic approach to
certain questions about these classes. For example, it makes it possible to
prove that $\per(\alpha)$ and $\ind(\alpha)$ have the same prime divisors by
using $p$-Sylow subgroups of Galois groups.

It was interesting for applications to other areas where period--index questions for Brauer groups arise, such as
globally over schemes, over topological spaces, over complex analytic spaces, and so on, to find methods of proof that
avoid the use of Galois groups. We were able to do this in a previous paper,~\cite{aw7}. We refer to
\cite{giraud}*{Section V.4} for a treatment of Azumaya algebras and the Brauer group in a locally ringed topos. There is a
natural bijective correspondence between Azumaya algebras of degree $n$ and $\PGL_n$-torsors. While the motivating
questions are phrased for Azumaya algebras, in practice we work with $\PGL_n$-torsors.

\begin{theorem}[\cite{aw7}]\label{thm:primes}
    Let $(X,\Oscr_X)$ be a connected locally-ringed topos and let
    $\alpha\in\Br(X,\Oscr_X)$. Then there exists an Azumaya $\Oscr_X$-algebra
    $\Ascr$ such that $[A]=\alpha$ and the prime divisors of $\per(\alpha)$ and $\deg(\Ascr)$
    coincide. In particular, the prime divisors of $\per(\alpha)$ and
    $\ind(\alpha)$ coincide.
\end{theorem}

Tam\'as Szamuely and Philippe Gille asked us recently if some other divisibility properties of
Brauer classes may be established in this generality.
We recall the following classical result over fields.
%
%

\begin{theorem}[\cite{gille-szamuely}*{Proposition 4.5.16},\cite{saltman}*{Theorem 5.7}]\label{thm:salt2}
    If $\alpha=[D]$ is in $\Br(k)$, where $k$ is a field and $D$ is the
    division algebra with class $\alpha$, and if $d=\ind(\alpha)$, then
    \begin{enumerate}
        \item   if $d=ab$ where $a$ and $b$ are relatively prime, then $D\iso
            E\otimes_k F$, where $E$ is a division algebra of degree $a$ and
            $E$ is a division algebra of degree $b$.
        \item   if $d=a_1\cdots a_r$ where the $a_i$ are relatively prime, then
            $D\iso E_1\otimes_k\cdots\otimes_k E_r$, where $E_i$ is a division
            algebra of degree $a_i$.
    \end{enumerate}
\end{theorem}

Saltman asked in~\cite{saltman} whether this type of result holds for Azumaya algebras and not just for
division algebras.
Our previous work and this paper, taken together, establish the maximum extent to which the
theory over fields generalizes to general contexts.
For example, we showed in~\cite{aw4,aw6} that Theorem~\ref{thm:salt2} fails for Azumaya
algebras over more general base schemes, even smooth affine schemes over $\CC$.

The point of this short note is to prove the following theorem, which
generalizes Theorem~\ref{thm:salt2} to the indices of Azumaya algebras:

\begin{theorem}\label{thm:main}
    Let $(X,\Oscr_X)$ be a connected locally-ringed topos, and let
    $\alpha=\alpha_1+\cdots+\alpha_t$ be the prime decomposition of a Brauer
    class $\alpha\in\Br(X,\Oscr_X)$, so that each $\per(\alpha_i)=p_i^{a_i}$ for
    distinct primes $p_1,\ldots,p_t$. Then,
    $$\ind(\alpha)=\ind(\alpha_1)\cdots\ind(\alpha_t).$$
\end{theorem}

That is, whereas prime decomposition cannot hold for general Azumaya algebras,
it does hold for the index. Using
Theorem~\ref{thm:main} and several facts about $p$-adic valuations of binomial
coefficients, we prove the next result. For division algebras, it was proved by
Saltman~\cite{saltman}*{Theorem 5.5}. Our proof is new over division algebras as well.

\begin{theorem}\label{thm:salt}
    Let $(X,\Oscr_X)$ be a connected locally-ringed topos. Suppose
    $\alpha\in\Br(X,\Oscr_X)$ is a Brauer class, and $d=\ind(\alpha)$ its index.
    Then:
    \begin{enumerate}
        \item   \label{j:1} $\ind(m\alpha)|\gcd(\binom{d}{m},d)$;
        \item   \label{j:2}  $\ind(m\alpha)=\ind(\alpha)$ if $m$ is prime to $d$;
        \item   \label{j:3} if $e=\gcd(m,d)$, then $\ind(m\alpha)$ divides $d/e$.
    \end{enumerate}
\end{theorem}

This answers the original question posed to us by Gille and Szamuely, which was
whether point (1) above holds in general. Note that the results of
Theorem~\ref{thm:salt} are straightforward to prove in the special case when
the index $d$ is a prime power, and we will do so in Lemma~\ref{lem:primepower}.

Theorems \ref{thm:main} and \ref{thm:salt} hold for the Brauer groups of arbitrary connected
schemes, or even algebraic stacks, for connected topological spaces, for connected complex analytic
spaces, for the topos of $G$-sets when $G$ is a discrete group, and so on.

\vspace{10pt}
\paragraph{Acknowledgments} We would like to thank Tam\'as Szamuely for
bringing the problem at hand to our attention as well as for detailed comments
on a draft. We would like also like to thank Anssi Lahtinen for useful
conversations. The first author thanks the NSF for its support via grant
DMS-1461847 and the Hausdorff Research Institute for Mathematics
for its hospitality during the period when this paper was conceived. Both
authors are grateful to
Banff International Research Station for providing a stunning setting in which
to write the paper and to the organizers of the workshop ``The Use of Linear
Algebraic Groups in Geometry and Number Theory'' for giving us a reason to be
there.

\section{Proof of Theorem~\ref{thm:salt} assuming Theorem~\ref{thm:main}}

Write $v_p(m)$ for the $p$-adic valuation of $m$, which is to say the largest power
of $p$ that divides $m$.
We will employ Kummer's theorem on $p$-adic valuations of binomial coefficients throughout.

\begin{theorem*}[Kummer's Theorem]\label{thm:kummer}
    Let $m$ and $n$ be nonnegative integers with $m\leq n$. Then,
    $v_p(\binom{n}{m})$ is the number of carries when $n-m$ is added to $m$ in
    base $p$, or equivalently it is the number of borrows when $m$ is
    subtracted from $n$ in base $p$.
\end{theorem*}

We will use the following special case of Kummer's Theorem several times.

\begin{corollary}\label{cor:kummer}
    Suppose that $0\leq r\leq s$ are integers, $\ell$ and $j$ are positive integers relatively prime to a prime
    number $p$, and $p^rj\leq p^s\ell$. Then,
    $v_p(\binom{p^s\ell}{p^rj})\geq s-r$, with equality if $p^r j<
    p^s$.
\end{corollary}

\begin{proof}
    The first non-zero entry in the $p$-adic expansion of $p^rj$ occurs exactly
    in the $r$th place, corresponding to the coefficient of $p^r$, and
    similarly the first non-zero entry of $p^sj$ is in the $s$th place.
    When subtracting $p^s\ell-p^rj$, there is a sequence of borrows from the
    $(r+1)st$ place to the $s$th place. Hence, there are at least $s-r$
    borrows.  If $p^rj<p^s$, no additional borrows occur.
\end{proof}

We make use of the exterior-power representations for $\PGL_n$ to deduce
Corollary~\ref{thm:salt}.
This was the main device of
\cite{aw7}. Given a $\PGL_n$-torsor with Brauer class $\alpha$,
and an integer $0 \le m \le n$, this produces a
$\PGL_{\binom{n}{m}}$-torsor with Brauer class $m\alpha$.
This construction is that of Proposition \ref{pr:LocRTop} in the
specific case of a Young diagram consisting of a single column.

\begin{lemma}\label{lem:primepower}
    Suppose that $\per(\alpha)=p^\nu$, where $p$ is a prime number.
    Then the following hold:
    \begin{enumerate}
    \item \label{i:1}$\ind(m\alpha)|\ind(\alpha)$ for all integers $m$;
    \item \label{i:2} $\ind(m\alpha) | \binom{\ind(\alpha)}{m}$ for all integers $m$ satisfying $1 \le m \le p^{\nu} -1$.
    \end{enumerate}
\end{lemma}

\begin{proof}
  The main result of~\cite{aw7} says that the indices of $\alpha$ and $m\alpha$ are powers of $p$.  Write
  $\ind(\alpha)=p^\sigma$.

  We prove item \eqref{i:2} first. By the definition of the index, there is a
  $\PGL_{p^\sigma\ell}$-torsor having Brauer class $\alpha$, where $\ell$ is prime to $p$. Taking the $m$-th exterior power, we
  produce a $\PGL_{\binom{p^\sigma\ell}{m}}$-torsor having Brauer class $m\alpha$. By Corollary~\ref{cor:kummer},
    \begin{equation*}
        v_p\left(\binom{p^\sigma\ell}{m}\right)=v_p\left(\binom{p^\sigma}{m}\right) = v_p\left( \binom{\ind(\alpha)}{m}
          \right),
    \end{equation*}
    since $m\leq p^\nu-1\leq p^\sigma-1$.

    We now prove item \eqref{i:1}. We may assume that $1 \le m \le p^\nu -1$, and apply part \eqref{i:2}. Kummer's
    theorem says that $v_p \left(\binom{\ind(\alpha)}{m}\right)$ is the number of
    carries when $m$ is added to $p^\sigma-m$ in base $p$. Both $m$ and $p^\sigma-m$ can be written in at most
    $\sigma$ base $p$ digits, and it follows that the number of carries is at most
    $\sigma$. Therefore $v_p(\ind(m\alpha)) \le \sigma$, as desired.
\end{proof}

\begin{proof}[Proof of Theorem~\ref{thm:salt} assuming Theorem~\ref{thm:main}]
    By Theorem~\ref{thm:main} and Lemma~\ref{lem:primepower},
    $\ind(m\alpha)|\ind(\alpha)$ for any Brauer class $\alpha$ and any $m$. If
    $m$ is prime to the period of $\alpha$, then \eqref{j:2} results immediately.

    We prove \eqref{j:1}. We know from Theorem~\ref{thm:main} and Lemma~\ref{lem:primepower} that $\ind(m\alpha)|d$, so
    we have to prove that $\ind(m\alpha)$ divides $\binom{d}{m}$. Let $p$ be a prime dividing $d$, and write
    $d=p^\sigma\ell$, where $\ell$ is prime to $p$. Let $\alpha_p$ denote the $p$-component of $\alpha$. Since
    $v_p(\ind(m\alpha))=v_p(\ind(m\alpha_p))$ by Theorem~\ref{thm:main}, it is enough to prove that $v_p(\ind(m\alpha_p))\leq
    v_p(\binom{d}{m})$.
    Let $m\equiv m'\pmod {p^\nu}$ with $0\leq m'<p^\nu$, where $\per(\alpha_p)=p^\nu$. Then,
    $v_p(\ind(m\alpha_p))=v_p(\ind(m'\alpha_p))$.  We know by Lemma~\ref{lem:primepower} and another invocation of
    Corollary~\ref{cor:kummer} that $v_p(\ind(m'\alpha_p))\leq
    v_p(\binom{p^\sigma}{m'})=v_p(\binom{p^\sigma\ell}{m'})$.
    So it suffices to show that $v_p(\binom{p^\sigma\ell}{m'})\leq v_p(\binom{p^\sigma\ell}{m})$. Since $v_p(m')=v_p(m)$
    by construction, this equality follows once again from Corollary~\ref{cor:kummer}.

    Finally, for \eqref{j:3}, note that by Theorem~\ref{thm:main}, it is enough to consider
    the case when $m=p^\delta$ and $d=p^\sigma$ are powers of the same prime
    $p$ and $\delta\leq\sigma$. Then, we know that $\ind(m\alpha)$ divides
    $\binom{p^\sigma}{p^\delta}$ by Lemma \ref{lem:primepower}. The
    $p$-adic valuation of this binomial coefficient is the number of carries
    when $p^\delta$ is added to $p^\sigma-p^\delta$ in the $p$-adic expansions
    of these numbers. There are $\sigma-\delta$ of these, and in the
    notation of \eqref{j:3}, we have $d/e=p^{\sigma-\delta}$.
\end{proof}

\section{Proof of Theorem \ref{thm:main}}

The method of proof of our theorem is to study certain morphisms
between projective general linear groups corresponding to symmetric powers. We arrived at these by considering morphisms
corresponding to more general Young tableaux, and some of the full theory of such morphisms is retained here in hope
that it may be useful in solving other problems.

We write $|\lambda|$ for the total number of boxes in the young diagram
$\lambda$. For all other conventions about Young diagrams and Young tableaux,
we refer to Fulton's book~\cite{fulton-tableaux}. 

To begin, we note that linear representations of $\GL_n$ corresponding to Young
diagrams can be defined integrally. A point that caused the authors some unease in the drafting is that the associated representations
of the symmetric group are not necessarily irreducible, as they are over the complex numbers. This is not important here,
however; all that is required is the existence of associated representations of $\GL_n$ on free modules, no reference is
made to irreducibility.

\begin{proposition}
  Let $R$ be a commutative ring. Let $\lambda$ be a Young diagram, let $n \ge 1$ be an integer and let
  $N$ denote the number of Young tableaux on $\lambda$ with entries in $\{1, \dots, n\}$. There is a map $\phi_\lambda(R):
  \GL_n(R) \to \GL_N(R)$, functorial in $R$, which fits in a functorially defined commutative square
  \begin{equation}
    \label{eq:1}
    \xymatrix{ \Gm(R) \ar[r] \ar^{x \mapsto x^{|\lambda|}}[d] & \GL_n(R) \ar^{\phi_\lambda}[d]\\
    \Gm(R) \ar[r] & \GL_N(R), }
  \end{equation}
  the horizontal maps being the inclusion maps of the subgroup of scalar invertible matrices.
\end{proposition}
\begin{proof}
  Write $V$ for $R^m$. We can construct a Schur module $V^\lambda$, as in
  \cite{fulton-tableaux}*{Chapter 8}. This construction is functorial in both the free $R$-module $V$ and the ring $R$, and
  $V^\lambda$ is equipped with a canonical $R$-linear $\End_R(V)$ action, and in particular a $\GL_n(R)$-action. The
  module $V^\lambda$ is a quotient of $V^{\tensor |\lambda|}$ by a certain module of $\GL_n(R)$-invariant relations, and
  therefore the reduction map $V^{\tensor |\lambda|} \to V^\lambda$ is compatible with the $\GL_n(R)$ action on each.

  The module $V^\lambda$ is a free $R$ module of dimension $N$ by
  \cite{fulton-tableaux}*{Chapter 8, Theorem 1}, and we
  therefore have a map $\phi_\lambda(R): \GL_n(R) \to \GL_N(R)$, which is moreover functorial in $R$.
  
  Finally, the reduction map $V^{\tensor |\lambda|} \to V^\lambda$ is $R$-linear, and the action of a scalar matrix $xI_n \in
  \GL_n(R)$ on $V^{\tensor |\lambda|}$ is by multiplication by $x^{|\lambda|}$, and it follows that $\phi_\lambda(xI_n)
  = x^{|\lambda|} I_N$ as asserted.
\end{proof}

\begin{proposition} \label{pr:LocRTop}
  Let $(X, \Oscr_X)$ be a locally ringed topos, let $\lambda$ be a Young
  diagram, let $m\ge 1$ be an integer, and let
  $N$ denote the number of Young tableaux on $\lambda$ with entries in $\{1, \dots, m\}$. There is a map of short exact
  sequences of group objects in $X$
  \begin{equation}
    \label{eq:2}
    \xymatrix{ 1 \ar[r] & \Gm \ar^{x \mapsto x^{|\lambda|}}[d] \ar[r] & \GL_m \ar^{\phi_\lambda}[d] \ar[r] & \PGL_m
        \ar^{\phi_\lambda}[d] \ar[r] & 1  \\ 
        1 \ar[r] & \Gm \ar[r] & \GL_N \ar[r] & \PGL_N \ar[r] & 1 }
  \end{equation}
  and in particular, there is a commutative diagram of cohomology groups in $X$
  \begin{equation}
    \label{eq:3}
    \xymatrix{ \Hoh^1(X, \PGL_n) \ar^{\phi_\lambda}[d] \ar[r] & \Br(X)
    \ar^{\times |\lambda|}[d] \ar@{^{(}->}[r] &
      \Hoh^2(X, \Gm) \ar^{\times |\lambda|}[d] \\
      \Hoh^1(X, \PGL_N) \ar[r] & \Br(X) \ar@{^{(}->}[r] & \Hoh^2(X, \Gm). }
  \end{equation}
\end{proposition}
\begin{proof}
  The objects $\GL_i$, including the special case $\Gm=\GL_1$, of $X$ are determined by the property that $\GL_i(U) =
  GL_i(\Oscr_X(U))$. The objects $\PGL_i$ are defined as the quotients $\GL_i/ \Gm$. Diagram \eqref{eq:2} therefore 
  requires only a commutative square 
\begin{equation*}
    \xymatrix{ \Gm(\Oscr_X(U) ) \ar[r] \ar^{x \mapsto x^{|\lambda|}}[d] & \GL_m(\Oscr_X(U)) \ar^{\phi_\lambda}[d]\\
    \Gm(\Oscr_X(U)) \ar[r] & \GL_N(\Oscr_X(U)), }
  \end{equation*}
  which is functorial in the ring $\Oscr_X(U)$. This is precisely the content of Diagram \eqref{eq:1}. We refer to
  \cite{giraud}*{IV 4.2.10} for the existence of the long exact sequence in cohomology in this generality.

  One may deduce that the map $\Gm \to \Gm$ given by $x \mapsto x^{|\lambda|}$ induces multiplication by $|\lambda|$ on
  cohomology groups by deriving these groups as \v{C}ech cohomology groups, among other ways. Diagram \eqref{eq:3} follows.
\end{proof}

The Azumaya algebra interpretation of Diagram \eqref{eq:3} is the following: given the data of a Young diagram $\lambda$
and a degree-$m$ Azumaya algebra $\Ascr$ on $X$, we may construct a degree-$N$ Azumaya $\Ascr'$ with the property that
$[\Ascr']= |\lambda| [\Ascr]$ in the Brauer group. The power of this construction lies in the great freedom we have in
our choice of $\lambda$.

To derive general results, however, we should like to have closed-form expressions for $N$. In
the remainder of the paper, the aim of which is to prove an existence result, we concentrate on one particular case in
which these closed-form expressions exist: that of Young diagrams of shape $(t)$, corresponding to $t$-fold symmetric
powers. In doing this, we abandon all pretence of minimality, being content to produce colossal representations which
happen to satisfy our requirements for prime factorization.

For a partition of shape $\lambda =(t)$, the number of Young tableaux on $\lambda$ with entries in
$\{1,\dots, m\}$---to wit, the dimension of the $t$-fold symmetric power of an $m$-dimensional vector space---is well known to be
\[N= \binom{t+m-1}{t}. \]

\begin{lemma}\label{lem:techLem2}
   Let $m\ge 2$ be a positive integer, let $p$ be a prime number such that $v_p(m) = s >
  0$. Let $\ell$ be an integer relatively prime to $p$. There exists some integer $r \ge 1$ satisfying the following
  three conditions:
  \begin{enumerate}
  \item \label{it1} $r \equiv 0 \pmod{\ell}$,
  \item \label{it2} $r \equiv 1 \pmod{p^s}$, and
  \item \label{it3} $v_p \left( \binom{r + m -1}{r} \right) = s$.
  \end{enumerate}
\end{lemma}

\begin{proof}
  Let $g$ be an integer exceeding $\log_p m$. Note that $g > s$, since $p^s | m$. Choose $r \ge 1$ such that 
  \begin{align*}
    r& \equiv 0 \pmod{\ell}, \\
    r &\equiv 1 \pmod{ p^g}.
  \end{align*}
  This $r$ satisfies conditions \eqref{it1} and \eqref{it2}. We calculate the $p$-adic valuation
  of  
\[ \binom{r+ m -1}{r} = \binom{r-1+ m}{m-1}\]
The $p$-adic expansion of $r-1$ has no terms in any position below the $g$-th place. The $p$-adic expansion of $m-1$ has
no terms in any position above the $g-1$-th place. It follows from this and
Kummer's theorem that the $p$-adic valuation of $\binom{r-1+ m}{m-1}$ agrees with that of $\binom{m}{m-1} = m$, which is
$s$. 
\end{proof}

\begin{proof}[Proof of Theorem~\ref{thm:main}]
  It suffices to prove that $v_{p_i}(\ind(\alpha))=v_{p_i}(\ind(\alpha_i))$; indeed, it is enough to prove
  this for $p=p_1$. Write $a=a_1$. Choose $\ell$ to be the non-$p$-primary part of $\per(\alpha)$, i.e.,
  $\ell = \prod_{i=2}^t p_i^{a_i}$.

 Write $s$ for $v_p(\ind(\alpha))$. We wish to show $s= v_p(\ind(\alpha_1))$. We first show the inequality
 $v_p(\ind(\alpha_1)) \ge s$ as follows. Take a representative
  $\Bscr$ for $\alpha_1$ for which $v_p(\deg(\Bscr)) = v_p(\ind(\alpha_1))$. Take a representative $\mathscr{C}$ for
  $\alpha_2 + \dots + \alpha_t$ for which $p \nmid \deg(\mathscr{C})$---this may be done using the main result of
  \cite{aw7} for instance. Then $\Bscr \tensor \mathscr{C}$ represents $\alpha$
  (see for example~\cite{aw7}*{Proposition 4}) and 
\[ s = v_p( \ind(\alpha)) \le v_p( \deg ( \Bscr \tensor \mathscr{C} )) =
  v_p(\deg ( \Bscr )) =  v_p(\ind(\alpha_1))\] as required.

  We now show the reverse inequality. There is some Azumaya algebra $\Ascr$, of degree $m$, representing
  $\alpha$ and such that $v_p(m) = v_p(\ind(\alpha))=s$.

  We apply Lemma \ref{lem:techLem2}, using $m, p, s$ and $\ell$ as above in order to obtain an integer $r$. Using the
  projective representation $\PGL_m \to \PGL_{N}$ given by Young diagram of the shape $(r)$, viz.~the $r$-fold symmetric
  power, we obtain a representative $\Ascr'$ for the Brauer class $r\alpha$ having degree $N=\binom{r + m -1}{r}$. Since
  $r \equiv 1 \pmod{p^a}$ and $r \equiv 0 \pmod{\ell}$, it follows that $r\alpha = \alpha_1$. In particular,
  $\ind(\alpha_1) | N$. We also know that $v_p(N) =  s$, so that $v_p(\ind(\alpha_1)) \le s$, as required.
\end{proof}

\section{Examples}

We end the paper with an example to show why it is necessary in the proof to consider
representations of projective general linear groups other than the exterior power representations of \cite{aw7}.
Let $P$ be a $\PGL_{36}$-torsor with Brauer class $\alpha=\alpha_2+\alpha_3$, these summands being of period $4$
and $9$ respectively.
We would like to show that $\alpha_2$ has index dividing $4$. Proceeding as in~\cite{aw7}, we might use the exterior algebra representations
$$\PGL_{36}\rightarrow\PGL_{\binom{36}{9}}$$ with class $\alpha_2$ and
$$\PGL_{36}\rightarrow\PGL_{\binom{36}{27}}$$ with class $-\alpha_2$ to find explicit representatives. However,
$v_{2}(\binom{36}{9})=v_2(\binom{36}{27})=4$, so these do not suffice to establish that $\ind(\alpha_2)$ is any smaller
than $16$.

Unwinding the proof of Theorem~\ref{thm:main}, we find we are asked to take $r$ such that $r \equiv 0 \pmod{9}$, $r
\equiv 1 \pmod{4}$ and such that $v_2\left(\binom{r+36-1}{r} \right) = 2$.
The smallest $r$ produced by the proof of Lemma~\ref{lem:techLem2} is $r=513$.

Once $r=513$ is chosen, one has $513\alpha=\alpha_2$, and
$513$-fold symmetric power gives a representation
\[ \PGL_{36} \to \PGL_{\binom{548}{513}} .\]
Hence, associated to any degree $36$ Azumaya algebra of class $\alpha$ there is an Azumaya
algebra of class $\alpha_2$ of degree $\binom{548}{513}$, which is divisible by $4=2^s$ but not by $8$.

We remark that this quantity, which is approximately
$2.3\times 10^{55}$, is simply the output of one particular construction that is known to
work in all cases. In fact, in the case of $36$ and the prime $2$, the $r=9$-fold (rather than the $513$-fold) symmetric
power may be taken instead, yielding the much smaller representation 
\[ \PGL_{36} \to \PGL_{N},\]
where
\[ N = \binom{44}{9} = 708930508 = 2^2\cdot 11\cdot 13\cdot 19\cdot 37\cdot 41\cdot 43. \]

Before settling on the current argument to prove Theorem \ref{thm:main}, we considered an argument based on diagrams of
the form $\lambda=(n,1)$.  These could be used to give a proof of Theorem~\ref{thm:main} along the same lines of the one
given here. The general procedure in that case for $36$ and $p=2$ produces the partition $\lambda=(260,1)$, which gives
an Azumaya algebra of degree $N$ around $1.14\times 10^{47}$ with $v_2(N)=2$ and class $\alpha_2$.

\end{document}